\documentclass[12pt]{amsart}
\usepackage{amsmath,amssymb}
\usepackage{eufrak}
\usepackage{amsthm}
\usepackage{enumitem}
\usepackage{setspace}
\usepackage{graphicx}
\usepackage{graphics}
\usepackage{color}
\usepackage{xcolor}
\usepackage{hyperref}
\usepackage{youngtab}
\usepackage{marginnote}
\usepackage{lscape}
\usepackage{faktor}
\usepackage{hyperref}
\usepackage[normalem]{ulem}
\usepackage{bm}
\usepackage{afterpage}
\usepackage{mathrsfs}
\usepackage[all]{xy}
\xyoption{matrix}
\xyoption{arrow}
\usepackage{tikz}
\usetikzlibrary{calc,shadings,patterns}
\usetikzlibrary{matrix}
\usetikzlibrary{arrows}
\usetikzlibrary{matrix}
\usepackage{verbatim}
\usepackage{multirow}
\usetikzlibrary{decorations.pathreplacing}
\usepackage{enumitem,pifont,xcolor}
\definecolor{myblue}{RGB}{0,29,119}

\newtheorem{theorem}{Theorem}[section]
\newtheorem{proposition}[theorem]{Proposition}

\theoremstyle{definition}
\newtheorem{definition}[theorem]{Definition}
\newtheorem{example}[theorem]{Example}

\newtheorem{remark}[theorem]{Remark}

\newtheorem*{theorem*}{Theorem}

\makeatletter
\@addtoreset{equation}{section}
\makeatother

\DeclareMathOperator{\End}{End}

\DeclareMathOperator{\findim}{fin.\! dim}

\DeclareMathOperator{\pdim}{p\,\!dim}
\DeclareMathOperator{\del}{del}

\newcommand{\cB}{{\mathcal B}}

\newcommand{\cS}{{\mathcal S}}

\providecommand{\AMS}{$\mathcal{A}$\kern-.1667em%
\lower.25em\hbox{$\mathcal{M}$}\kern-.125em$\mathcal{S}$}

\begin{document}

\title[Delooping Level of Nakayama algebras]{Delooping Level of Nakayama algebras}

\author{Emre SEN}

\address{Department of Mathematics, University of Iowa, Iowa City, IA}
\email{emre-sen@uiowa.edu}

\maketitle
\begin{abstract}  We give another proof of the recent result of Ringel, which asserts equality between the finitistic dimension and delooping level of Nakayama algebras. The main tool is syzygy filtration method introduced in \cite{sen2019}. In particular, we give characterization of the finitistic dimension one Nakayama algebras.
\end{abstract}

\tableofcontents

\section{Introduction}
Let $A$ be an artin algebra. The delooping level of finitely generated $A$-module $M$ is defined by Gelinas in \cite{gel} as:
\begin{align}\label{defdel1}
\del M=\min\left\{d\,\vert\,\Omega^d(M)\subset {_A}A\oplus\Omega^{d+1}(M'),\quad M'\subset \text{mod-}A\right\}
\end{align}
and $\del M=\infty$ if such a $d$ does not exist. 
The delooping level of algebra $A$ is the maximum of delooping levels of simple $A$-modules: 
\begin{align}\label{defdel2}
\del A=\max\left\{\del S\,\vert \,S\, \text{is simple}\, A\,\text{module}\right\}
\end{align}

The finitistic dimension of algebra $A$ is given by 
\begin{align*}
\findim A=\sup\left\{\pdim M\,\vert\,M\subset\text{mod-}A\right\}
\end{align*}

An open problem in representation theory of artin algebras is the finitistic dimension conjecture which states that the finitistic dimension of $A$ is finite. Relationship between the finitistic dimension and the delooping level is:
\begin{align*}
\findim A^{op}\leq \del A
\end{align*}
which makes the delooping level an interesting homological measure, where $A^{op}$ is opposite algebra \cite{gel}.

For the rest of the paper, we focus our attention to Nakayama algebras i.e. all indecomposable modules are uniserial.  Ringel proved the following theorem in \cite{rin2020del}:

\begin{theorem}\label{thm}
The finitistic dimensions and the delooping levels of Nakayama algebra and its opposite algebra are same.
\end{theorem}

In this short note, we give another proof of the theorem by using syzygy filtration method. The keystone is the following reduction:
\begin{gather*}
\del \Lambda=\del \bm{\varepsilon}(\Lambda)+2
\end{gather*}
which we show in proposition \ref{prop1}, where ${\bm{\varepsilon}}(\Lambda)$ is syzygy filtered algebra (see def. \ref{filteredalg}). This helps us to make mathematical induction on the homological dimensions, so it is enough to study Nakayama algebras with small finitistic dimensions (see prop. \ref{findim1}). In the next section, we recall some definitions  and results about syzygy filtrations. In the last section, we state a characterization of Nakayama algebras having small finitistic dimensions and the proof of the main theorem.

\subsection{Acknowledgments}
We are deeply thankful to Prof. Ringel for devoting appendices B and C in his work to our $\bm\varepsilon$-construction and elucidating its difference from the other methods as well as to Prof. Igusa and Prof. Todorov for encouraging us to post this paper.

\section{Preliminaries on Syzygy Filtration}\label{sectionprelim}
Here, we recall some basic definitions and construction of syzygy filtered algebra ${\bm\varepsilon}(\Lambda)$ which will be used in the section \ref{sectionproof}. For details we refer to papers \cite{sen2018} and \cite{sen2019}.
\subsection{Nakayama Algebras via Defining Relations}
Consider the irredundant system of relations $\alpha_{k_{2i}}\ldots\alpha_{k_{2i-1}}=0$ where $1\leq i\leq  r$ and $k_{f}\in\left\{1,2,\ldots,n\right\}$ for a cyclic oriented quiver $Q$ where each arrow $\alpha_i$, $1\leq i\leq n-1$ starts at the vertex $i$ and ends at the vertex $i+1$ and $\alpha_n$ starts at vertex $n$ and ends at vertex $1$. It follows that bound quiver algebra $\Lambda=kQ/I$ where $I$ is the irredundant system of relations, is cyclic Nakayama algebra. Explicitly $I$ is generated by relations:
\begin{gather}\label{relations}
\alpha_{k_2}\ldots\alpha_{k_1+1}\alpha_{k_1}\ \ =0 \\
\alpha_{k_4}\ldots\alpha_{k_3+1}\alpha_{k_3}\ \ =0\nonumber \\
\vdots \nonumber\\
\alpha_{k_{2r-2}}\ldots\alpha_{k_{2r-3}+1}\alpha_{k_{2r-3}}=0\nonumber\\
\alpha_{k_{2r}}\ldots\alpha_{k_{2r-1}+1}\alpha_{k_{2r-1}}=0\nonumber
\end{gather}
where $..<k_1<k_3<\ldots<k_{2r-1}<k_1<\ldots$ is cyclically ordered \cite{sen2018}. 
\subsection{Projective and Injective Modules}\label{projectives} We can describe indecomposable projective-injective modules by relations. They are:
\begin{align}
P_{k_1+1}=I_{k_4},\quad P_{k_3+1}=I_{k_6},\quad \dots \quad P_{k_{2r-3}+1}=I_{k_{2r}},\qquad P_{k_{2r-1}+1}=I_{k_{2}}
\end{align}
Projective modules can be characterized by their socles:

\begin{align*}P_{k_{1}}\hookrightarrow\ldots\hookrightarrow  P_{(k_{2r-1})+1}=I_{k_{2}} \quad\text{ have simple } S_{k_{2}} \text{  as their socle}\\
P_{k_3}\hookrightarrow\ldots\hookrightarrow P_{k_1+1}=I_{k_4}  \quad\text{ have simple } S_{k_4} \text{  as their socle}\\
P_{k_5} \hookrightarrow\ldots\hookrightarrow  P_{k_3+1}=I_{k_6}\quad \text{ have simple } S_{k_6} \text{  as their socle}\\
\vdots\qquad\qquad\qquad\qquad\qquad\qquad\qquad\\
P_{k_{2r-1}}\hookrightarrow\ldots\hookrightarrow P_{(k_{2r-3})+1}=I_{k_{2r}}\quad  \text{ have simple } S_{k_{2r}} \text{  as their socle}
\end{align*}
Notice that the minimal length projectives in their classes are:
\begin{align}\label{minproj}
P_{k_1},P_{k_3},\ldots,P_{k_{2r-1}}.
\end{align}

Analogously, injective modules can be characterized by their tops:
\begin{align*}
 P_{(k_{2r-1})+1}=I_{k_{2}}\twoheadrightarrow\ldots\twoheadrightarrow I_{k_{2r}+1}  \quad\text{ have simple } S_{k_{2r-1}+1}\text{  as their top}\\
 P_{k_1+1}=I_{k_4}\twoheadrightarrow\ldots\twoheadrightarrow I_{k_2+1}  \quad\quad\text{ have simple } S_{k_1+1} \text{  as their top}\\
  P_{k_3+1}=I_{k_6}\twoheadrightarrow\ldots\twoheadrightarrow I_{k_4+1}  \quad\quad\text{ have simple } S_{k_3+1} \text{  as their top}\\
\vdots\qquad\qquad\qquad\qquad\quad\quad\quad\quad\quad\quad\\
P_{(k_{2r-3})+1}=I_{k_{2r}} \twoheadrightarrow\ldots\twoheadrightarrow I_{k_4+1}  \quad\quad\text{ have simple } S_{k_{2r-3}+1} \text{  as their top}
\end{align*}

\subsection{Syzygy Filtered Algebra}\cite{sen2019}
Let $\cS(\Lambda)$ be the complete set of representatives of socles of projective modules over $\Lambda$:
\begin{align}\label{defS}
\cS(\Lambda)=\left\{S_{k_2}, S_{k_4},\ldots,S_{k_{2r}}\right\}
\end{align}
 Similarly, let $\cS'(\Lambda)$ be the complete set of representatives of simple modules such that they are indexed by one cyclically larger indices of $\cS(\Lambda)$:
\begin{align}\label{defSprime}
\cS'(\Lambda)=\left\{S_{k_{2}+1}, S_{k_4+1},\ldots,S_{k_{2r}+1}\right\}
\end{align}

Now we define the following base set $\cB(\Lambda)$ and its $\Lambda^{op}$ analogue $\nabla(\Lambda)$:
\begin{align}\label{baseset}
\cB(\Lambda):=\left\{ \Delta_1\cong\begin{vmatrix}
    S_{k_{2r}+1} \\
    \vdots  \\
    S_{k_{2}}
\end{vmatrix}\!, \Delta_2\cong\begin{vmatrix}
    S_{k_{2}+1}  \\
    \vdots  \\
   S_{k_{4}}
\end{vmatrix}\!,..,\Delta_j\cong\begin{vmatrix}
   S_{k_{2(j-1)}+1}  \\
    \vdots  \\
    S_{k_{2j}}
\end{vmatrix}\!,..,\Delta_r\cong \begin{vmatrix}
   S_{k_{2r-2}+1}  \\
    \vdots  \\
    S_{k_{2r}}
\end{vmatrix}\!\right\}
\end{align}

\begin{align}\label{nablaset}
{\nabla}\left(\Lambda\right):=\left\{ \nabla_1\cong\begin{vmatrix}
    S_{k_{1}+1} \\
    \vdots  \\
    S_{k_{3}}
\end{vmatrix}\!, \nabla_2\cong\begin{vmatrix}
    S_{k_{3}+1}  \\
    \vdots  \\
   S_{k_{5}}
\end{vmatrix}\!,..,\nabla_j\cong\begin{vmatrix}
   S_{k_{2j-1}+1}  \\
    \vdots  \\
    S_{k_{2j+1}}
\end{vmatrix}\!,..,\nabla_r\cong \begin{vmatrix}
   S_{k_{2r-1}+1}  \\
    \vdots  \\
    S_{k_{1}}
\end{vmatrix}\!\right\}.
\end{align}
\begin{remark} In the terminology of Ringel \cite{rin2020del} Appendix C, $\cB(\Lambda)$ modules are the first syzygy modules of valley modules. Also he shows how some of these can be interpreted by other methods.
\end{remark}

\begin{definition}\label{filteredalg} \cite{sen2019}
Let $\Lambda$ be cyclic Nakayama algebra. Syzygy filtered algebra $\bm{\varepsilon}(\Lambda)$ is:
\begin{align}
\bm{\varepsilon}(\Lambda):=\End_{\Lambda}\left(\bigoplus\limits_{S\in \cS'(\Lambda)}P(S)\right)
\end{align} 
$d$th syzygy filtered algebra $\bm{\varepsilon}^d(\Lambda)$ is:
\begin{align}
\bm{\varepsilon}^d(\Lambda):=\End_{\bm{\varepsilon}^{d-1}(\Lambda)}\left(\bigoplus\limits_{S\in\cS'(\bm{\varepsilon}^{d-1}(\Lambda))}P(S)\right)
\end{align}
provided that $\bm{\varepsilon}^{d-1}(\Lambda)$ is cyclic non self-injective Nakayama algebra.
\end{definition}

For cyclic Nakayama algebras of infinite global dimension, a nice homological measure is $\varphi$-dimension which attains only even numbers (see \cite{sen2018} thm A). Here we briefly recall definition of $\varphi$-dimension:
\begin{definition}\label{varphi}
For a given $A$-module $M$, let $\varphi\left(M\right)$ be defined in \cite{it} as:
$$\varphi(M):=min\{t\ |\ rank\left(L^t\langle add M\rangle\right)=rank\left(L^{t+j}\langle add M\rangle\right)\text{ for }\forall j\geq 1\}$$ where $L[M]:=[\Omega M]$ in $K_0$ group.  
$\varphi$-dimension of algebra $A$ is:
$$\varphi\dim(A):=sup\{\varphi(M)\ |\ M \in mod A\}.$$ 
\end{definition}

\begin{remark}\label{remarkfacts} We collect and summarize some useful results about $\varphi$-dimension and syzygy filtration from \cite{sen2018} and \cite{sen2019}.
\begin{enumerate}[label=\arabic*.]
\item\label{listfiltration}  The second and higher syzygies of $\Lambda$ modules have a unique $\cB(\Lambda)$ filtration.
\item\label{listcategory}  Category of $\cB(\Lambda)$ filtered $\Lambda$-modules is equivalent to category of ${\bm\varepsilon}(\Lambda)$-modules.
\item\label{listnakayama} ${\bm\varepsilon}(\Lambda)$ is Nakayama algebra.
\item\label{listreduction} $\bm\varepsilon$-construction reduces the following homological dimensions by exactly two: $\varphi$-dimension, finitistic, dominant and Gorenstein dimensions.
\item\label{listreduction2} If global dimension of $\Lambda$ is infinite then there exists $d$ such that ${\bm\varepsilon}^d(\Lambda)$ is selfinjective Nakayama algebra and $\varphi\dim{\bm\varepsilon}^{d-1}(\Lambda)=2$. 
\item\label{listdifference} The difference $\varphi\dim\Lambda-\findim\Lambda\leq 1$. In particular, $\findim\Lambda=1$ or $\findim\Lambda=2$ imply $\varphi\dim\Lambda=2$.

\end{enumerate}
 \end{remark}

\section{Proof of the Main Theorem} \label{sectionproof}
In this section, first we prove the reduction formula. Then we give characterization of the finitistic dimension one Nakayama algebras in terms of defining relations. Together with study of periodic modules, we show equalities of the delooping level and the finitistic dimension while $\varphi$-dimension is two. At the end, we give proof of the main theorem.

\begin{proposition}\label{prop1} Let $\Lambda$ be cyclic Nakayama algebra with $\varphi\dim\Lambda\geq 3$. Then:
\begin{gather}
\del \Lambda=\del \bm{\varepsilon}(\Lambda)+2
\end{gather}
\end{proposition}
\begin{proof}
Let $S$ be the simple $\Lambda$ module with $\del S=d\geq 3$. By definition \ref{defdel1}, there exists $\Lambda$ module $M$ satisfying $\Omega^d(S)\subset {}_{\Lambda}\Lambda\oplus \Omega^{d+1}(M)$, provided that $d$ is minimal.

This implies:
\begin{gather*}
\Omega^{d-2}(\Omega^2(S))\subset {}_{\Lambda}\Lambda\oplus \Omega^{d-1}(\Omega^2(M))\\
\Omega^{d-2}(S')\subset {}_{ \bm{\varepsilon}(\Lambda)} \bm{\varepsilon}(\Lambda)\oplus \Omega^{d-1}(M'))
\end{gather*}
where $S'=\Omega^2(M)$, $M'=\Omega^2(M)$ are $ \bm{\varepsilon}(\Lambda)$  modules. $d-2$ is minimal otherwise $d$ would not be delooping level of simple $\Lambda$-module $S$.
\end{proof}

This observation enables us to make mathematical induction on the delooping level of $\Lambda$. So, it is enough to analyze the cases where $\del\Lambda$ is one or two. First we examine the case with the finitistic dimension one.

\begin{proposition}\label{findim1}
The following are equivalent for a cyclic Nakayama algebra $\Lambda$ given by irredundant system of relations [\ref{relations}]:
\begin{enumerate}[label=\arabic*.]
\item\label{item1} $\findim\Lambda=1$
\item\label{item2} $\left\{k_1,k_3,\ldots,k_{2r-1}\right\}=\left\{k_2,k_4,\ldots,k_{2r}\right\}$
\item\label{item3} $\cB(\Lambda)=\nabla(\Lambda)$
\end{enumerate}
\end{proposition}
\begin{proof}
 
$1\implies 2$: Before starting the proof, by the result \ref{remarkfacts}(\ref{listdifference}), we can assume that $\varphi\dim\Lambda=2$ which bounds the finitistic dimension by $2$. Let $S$ be a simple module with $\pdim S=1$. Because the finitistic dimension of $\Lambda$ is one, $S$ cannot be a socle of projective module. Otherwise we would get the resolution:
\begin{center}
$\xymatrixcolsep{6pt}\xymatrixrowsep{6pt}
\xymatrix{&P\ar[rr]\ar[rd]&&P'\ar[rr]&&\faktor{P'}{S}\\
rad P\ar[ru]&&S\ar[ru]&&
}$
\end{center} and $\pdim\faktor{P'}{S} =2$ would make $\findim\Lambda=2$. Therefore the only possible choices for the socle of projective modules are the tops of minimal length projective modules (\ref{minproj}). By the system of relations, their indices form the set $\left\{k_1,k_3,\ldots,k_{2r-1}\right\}$. On the other hand, socles of projective modules are $S_{k_2}, S_{k_4},\ldots, S_{k_{2r}}$ (see def.\ref{defS}) and indices are $\left\{k_2,k_4,\ldots,k_{2r}\right\}$. Two sets have to be equal.

$2\implies 3$. Equality of sets $\left\{k_1,k_3,\ldots,k_{2r-1}\right\}=\left\{k_2,k_4,\ldots,k_{2r}\right\}$ implies equality of $\cB(\Lambda)=\nabla(\Lambda)$ which simply follows from definitions \ref{baseset} and \ref{nablaset}.

$3\implies 1$ Equality of $\cB(\Lambda)=\nabla(\Lambda)$ implies that all projective-injective modules have $\cB(\Lambda)$-filtration because their tops are from the set $S'(\Lambda)$ (\ref{defSprime}). Indeed, those are the only projective modules having $\cB(\Lambda)$ filtration due to the cyclic ordering of indices. Assume to the contrary that there is a module $M$ with $\pdim M\geq 2$. Therefore $\Omega^{\pdim M}(M)$ is a projective module. Since it is either the second syzygy or higher syzygy, it has $\cB(\Lambda)$ filtration (see remark \ref{remarkfacts} \ref{listfiltration}). Therefore it has to be projective-injective module, which is impossible by length considerations i.e. projective-injective module cannot be a syzygy module.
\end{proof}
\begin{remark} This also implies that $\findim\Lambda=1\iff\findim\Lambda^{op}=1$ which simply follows by the item \ref{item3} of proposition \ref{findim1} and $\findim\Lambda=2\iff\findim\Lambda^{op}=2$ combined by remark \ref{remarkfacts}(\ref{listdifference}). Importance and usage of defining relations is also stressed by Ringel in \cite{rin2020del}.
\end{remark}
\begin{example} Just to show how we interpret item (\ref{item2}) of proposition \ref{findim1}, let  $1,3,5$ be indices of minimal projective modules of algebra of rank $5$. We list some cases:
\begin{enumerate}[label=\roman*)]
\item $P_1=\begin{matrix}
1\\2\\3
\end{matrix}\quad P_2=\begin{matrix}
2\\3\\4\\5
\end{matrix}$
$\quad P_3=\begin{matrix}
3\\4\\5
\end{matrix}$
$\quad P_4=\begin{matrix}
4\\5\\1
\end{matrix}$
$\quad P_5=\begin{matrix}
5\\1
\end{matrix}$\vspace{0.5cm}
\item $P_1=\begin{matrix}
1\\2\\3\\4\\5
\end{matrix}\quad P_2=\begin{matrix}
2\\3\\4\\5\\1
\end{matrix}$
$\quad P_3=\begin{matrix}
3\\4\\5\\1
\end{matrix}$
$\quad P_4=\begin{matrix}
4\\5\\1\\2\\3
\end{matrix}$
$\quad P_5=\begin{matrix}
5\\1\\2\\3
\end{matrix}$\vspace{0.5cm}
\item $P_1=\begin{matrix}
1\\2\\3\\4\\5\\1
\end{matrix}\quad P_2=\begin{matrix}
2\\3\\4\\5\\1\\2\\3
\end{matrix}$
$\quad P_3=\begin{matrix}
3\\4\\5\\1\\2\\3
\end{matrix}$
$\quad P_4=\begin{matrix}
4\\5\\1\\2\\3\\4\\5
\end{matrix}$
$\quad P_5=\begin{matrix}
5\\1\\2\\3\\4\\5
\end{matrix}$
\end{enumerate}
In all the cases $\cB(\Lambda)=\left\{\begin{matrix}
4\\5
\end{matrix},\,\,\begin{matrix}
1
\end{matrix},\,\,\begin{matrix}
2\\3
\end{matrix}\right\}$, it is easy to verify that all of these algebras have the finitistic dimension one.
\end{example}

\begin{remark}\label{remark1}  From definitions, we can deduce that $\del\Lambda\leq \varphi\dim\Lambda$. Before proceeding our analysis, we give definition of periodic module: $M$ is called periodic if there exists $t$ such that $\Omega^t(M)=M$. In details, there are two types of simple modules:
\begin{itemize}
\item Projective dimension of $S$ is finite, so $\pdim S\leq \varphi\dim\Lambda$.
\item Projective dimension of $S$ is unbounded, then $\Omega^{\varphi\dim\Lambda}(S)$ has to be periodic module, so there exists another periodic module $M$ such that :
\begin{align*}
\Omega^{\varphi\dim\Lambda}(S)\cong\Omega(M)
\end{align*}
\end{itemize}
\end{remark}

\begin{proposition}\label{prop2}
Let $\Lambda$ be cyclic Nakayama algebra with $\varphi\dim\Lambda=2$. Then:
\begin{center}
\begin{itemize}
\item $\findim\Lambda=1 \iff \del\Lambda=1$
\item $\findim\Lambda=2\iff \del\Lambda=2$
\end{itemize}
\end{center}

\end{proposition}

\begin{proof}
It is enough to show the first item, since both of $\del\Lambda$ and $\findim\Lambda$ bounded by $2$ which follows from remarks \ref{remarkfacts}(\ref{listdifference}) and \ref{remark1}. First we show that  the implication $\findim\Lambda=1\implies\del\Lambda=1$. Assume that $\findim\Lambda=1$. By the remark \ref{remark1}, possibilities are:
\begin{enumerate}[label=\roman*)]
\item Projective dimension of $S$ is one, therefore $\del S=1$
\item $S$ itself can be a periodic module which makes $\del S=0$.
\item $S$ is not periodic but $\Omega(S)$ is periodic module i.e. non-periodic tops of a minimal projectives \ref{minproj}. We conclude $\del S=1$.
\end{enumerate}
As a result, maximum of $\del S$ is one. 

Now, we consider the other direction. Since $\del\Lambda=\max_S\{\del S\}$, we have two cases:
\begin{enumerate}[label=\roman*)]
\item $\pdim S=1$, 
\item $\Omega(S)$ is periodic module. 
\end{enumerate}
If projective dimension of a simple module is one, it cannot be top of minimal projective. The second item forces that the first syzygies of tops of minimal projectives have $\cB(\Lambda)$ filtration. Therefore they are the indices of simple modules of the set $S'(\Lambda)$. On the other hand, the indices of the tops of the first syzygies are $\{k_1+1,k_3+1,\ldots,k_{2r-1}+1\}$ in this case. We get:
$\{k_2+1,k_4+1,\ldots,k_{2r}+1\}=\{k_1+1,k_3+1,\ldots,k_{2r-1}+1\}$ which is equivalent to item \ref{item2} in proposition \ref{findim1}. This finishes the proof.
\end{proof}

In the proof, we did not need to analyze the case $\del\Lambda=2$. Actually, it is equivalent to existence of a simple module $S$ which is not a socle of any projective module and in particular $\Omega^2(S)$ is periodic but $\Omega(S)$ is not. \\

It is clear that if global dimension is finite, $\del\Lambda=\findim\Lambda$. Hence we consider infinite global dimensional case. Moreover, we can exclude the equalities of the left and right finitistic dimensions, since we proved them in \cite{sen2019}. Now we can use induction to prove the remaining cases of the main result:
\begin{proof}[Proof of Theorem \ref{thm}] 
Let $\Lambda$ be cyclic (not selfinjective) Nakayama algebra of infinite projective dimension. Then there exists $d$ (see remark \ref{remarkfacts}\ref{listreduction2}) such that:
\begin{align}
1\leq \findim\bm{\varepsilon}^d(\Lambda)\leq 2
\end{align}
Because ${\bm\varepsilon}^{d}(\Lambda)$ is cyclic Nakayama algebra (see remark \ref{remarkfacts} \ref{listnakayama}), we can use proposition \ref{prop2} to get:
\begin{align}
1\leq \findim\bm{\varepsilon}^d(\Lambda)=\del \bm{\varepsilon}^d(\Lambda) \leq 2
\end{align}
Now, by proposition \ref{prop1}, we obtain the desired equality $\findim\Lambda=\del \Lambda$ because:
\begin{align}
1+2d\leq \findim\Lambda=\del \Lambda \leq 2+2d.
\end{align}
\end{proof}

\bibliographystyle{alpha}

 \end{document}